\documentclass[11pt]{article}
\usepackage{amsmath,amssymb,amsthm,color}
\usepackage{comment}

\newtheorem{theorem}{Theorem}[section]
\newtheorem{corollary}{Corollary}[section]
\newtheorem{lemma}{Lemma}[section]
\newtheorem{claim}{Claim}[section]

\def\R{{\mathfrak R}\, }
\def\D{{\mathfrak D}\, }
\def\A{{\mathfrak A}\, }

\begin{document}
\begin{center}{\bf \LARGE $*$-Jordan-type maps on $C^*$-algebras}\\
\vspace{.2in}
\noindent {\bf Bruno Leonardo Macedo Ferreira}\\
{\it Federal University of Technology,\\
Avenida Professora Laura Pacheco Bastos, 800,\\
85053-510, Guarapuava, Brazil.}\\
e-mail: brunoferreira@utfpr.edu.br

\vspace{.2in}
and
\\
\vspace{.2in}

\noindent {\bf Bruno Tadeu Costa}\\
{\it Federal University of Santa Catarina,\\
Rua Jo\~{a}o Pessoa, 2750,\\   
89036-256, Blumenau, Brazil.}\\
e-mail: b.t.costa@ufsc.br

\end{center}
\begin{abstract} 
Let $\A$ and $\A'$ be two  
$C^*$-algebras with identities $I_{\A}$ and $I_{\A'}$, respectively, and $P_1$ and $P_2 = I_{\A} - P_1$ nontrivial 
projections in $\A$. In this paper we study the  characterization of multiplicative $*$-Jordan-type maps. In particular, 
if $\mathcal{M}$ is a factor von Neumann algebra then every bijective unital multiplicative $*$-Jordan-type maps are $*$-ring isomorphisms.
\end{abstract}

\section{Introduction and Preliminaries}

Historically, the study of additivity of maps have received a fair amount of attention of mathematicians. The first quite surprising result is due to Martindale who established a condition on a ring such that multiplicative bijective maps are all additive \cite[Theorem]{Mart}. Besides, over the years several works have been published considering different types of associative and non-associative algebras. Among them we can mention \cite{Fer, Fer1, Fer2, Fer3, Fer4, bruth, FerGur1, chang}.
In order to add new ingredients to the study of additivity of maps, many researches have devoted themselves to the investigation of two new products, presented by Bre$\check{s}$ar and Fo$\check{s}$ner in \cite{brefos1, brefos2}, where the definition is as follows: for $A, B \in R$, where $R$ is a $*-$ring, we denote by $A\bullet B = AB+BA^{*}$ and $[A, B]_{*} = AB-BA^{*}$ the $*$-Jordan product and the $*$-Lie product, respectively. In \cite{LiLuFang}, the authors proved that a map $\Phi$ between two factor von Newmann algebras is a $*$-ring isomorphism if and only if $\Phi(A\bullet B) = \Phi(A)\bullet \Phi(B)$. In \cite{Ferco}, Ferreira and Costa extended these new products and defined two other types of applications, named multiplicative $*$-Lie n-map and multiplicative $*$-Jordan n-map and used it to impose condition such that a map between $C^*$-algebras is a $*$-ring isomorphism. With this picture in mind, in this article we will discuss when a multiplicative $*$-Jordan n-map is a $*$-ring isomorphism. As a consequence of our main result, we provide an application on von Neumann algebras, factor von Neumann algebras and prime algebras. Furthermore, we generalize the Main Theorem in \cite{LiLuFang}.

Let us define the following sequence of polynomials, as defined in \cite{Ferco}: 
$$q_{1*}(x) = x\, \,  \text{and}\, \,  q_{n*}(x_1, x_2, \ldots , x_n) = \left\{q_{(n-1)*}(x_1, x_2, \ldots , x_{n-1}) , x_n\right\}_{*},$$
for all integers $n \geq 2$. Thus, $q_{2*}(x_1, x_2) = \left\{x_1, x_2\right\}_{*}, \ q_{3*} (x_1, x_2, x_3) = \left\{\left\{x_1, x_2\right\}_{*} , x_3\right\}_{*}$, etc. Note that $q_{2*}$ is the product introduced by Bre$\check{s}$ar and Fo$\check{s}$ner \cite{brefos1, brefos2}. Then, using the nomenclature introduced in \cite{Ferco} we have a new class of maps (not necessarily additive): $\varphi : \R \longrightarrow \R'$ is a \textit{multiplicative $*$-Jordan $n$-map} if
\begin{eqnarray*}\label{ident1}
\\&&\varphi(q_{n*} (x_1, x_2, . . . ,x_n)) =  q_{n*} (\varphi(x_1), \varphi(x_2), . . . , \varphi(x_i), . . .,\varphi(x_n)),
\end{eqnarray*}
where $n \geq 2$ is an integer.
Multiplicative $*$-Jordan $2$-map, $*$-Jordan $3$-map 
and $*$-Jordan $n$-map are collectively referred to as \textit{multiplicative $*$-Jordan-type maps}.

By a $C^*$-algebra we mean a complete normed complex algebra (say $\A$) endowed with a conjugate-linear algebra involution $*$, 
satisfying $||a^*a||=||a||^2$ for all $a \in \A$. Moreover, a $C^*$-algebra is a {\it prime} $C^*$-algebra if $A\A B = 0$ for 
$A,B \in \A$ implies either $A=0$ or $B=0$.

We find it convenient to invoke the noted Gelfand-Naimark theorem that state a $C^*$-algebra $\A$ is $*$-isomorphic to a
$C^*$-subalgebra $\D \subset \mathcal{B}(\mathcal{H})$, where $\mathcal{H}$ is a Hilbert space. So from now on we shall
consider elements of a $C^*$-algebra as operators.

Let be $P_1$ a nontrivial projection in $\A$ and $P_2 = I_{\A} - P_1$ where $I_{\A}$ is the identity of $\A$. Then $\mathfrak{A}$ has a decomposition
$\mathfrak{A}=\mathfrak{A}_{11}\oplus \mathfrak{A}_{12}\oplus
\mathfrak{A}_{21}\oplus \mathfrak{A}_{22},$ where
$\mathfrak{A}_{ij}=P_{i}\mathfrak{A}P_{j}$ $(i,j=1,2)$.

The following two claims play a very important role in the further development of the paper. By definition of involution clearly we get
 
\begin{claim}\label{obs}
$*(\A_{ij}) \subseteq \A_{ji}$ for $i,j \in \left\{1,2\right\}.$
\end{claim}

\begin{claim}\label{Claim2} Let $\A$ and $\A'$ be two $C^*$-algebras and $\varphi: \A \rightarrow \A'$ a bijective map which satisfies
\begin{eqnarray*}
 &&\varphi(q_{n*}(I_\A,..., I_\A, A,B)) = q_{n*}(\varphi(I_\A),...., \varphi(I_\A), \varphi(A),\varphi(B))
\\&and&
\\&&\varphi(q_{n*}(P, ..., P, A, B)) = q_{n*}(\varphi(P), ..., \varphi(P),\varphi(A),\varphi(B)),
\end{eqnarray*}
 for all $A, B \in \A$ and $P \in \left\{P_1, P_2\right\}$. Let $X,Y$ and $H$ be in $\A$ such that $\varphi(H) = \varphi(X) + \varphi(Y)$. Then,
 given $Z \in \A$,
$$
\begin{aligned}
\varphi(q_{n*}(T,...,T,H,Z)) &= \varphi(q_{n*}(T,...,T,X,Z)) \\
                                                   &+ \varphi(q_{n*}(T,...,T,Y,Z))
\end{aligned}
$$
and
$$
\begin{aligned}
\varphi(q_{n*}(T,...,T,Z,H)) &= \varphi(q_{n*}(T,...,T,Z,X)) \\
                                                   &+ \varphi(q_{n*}(T,...,T,Z,Y))
\end{aligned}
$$
for $T = I_\A$ or $T = P$.

\end{claim}

\begin{proof}
Using the definition of $\varphi$ and multilinearity of $q_{n*}$ we obtain
$$
\begin{aligned}
\varphi(q_{n*}(T,...,T,H,Z)) &= q_{n*}(\varphi(T),...,\varphi(T),\varphi(H),\varphi(Z)) \\
                             &= q_{n*}(\varphi(T),...,\varphi(T),\varphi(X)+\varphi(Y),\varphi(Z)) \\
                             &= q_{n*}(\varphi(T),...,\varphi(T),\varphi(X),\varphi(Z)) \\
                             &+ q_{n*}(\varphi(T),...,\varphi(T),\varphi(Y),\varphi(Z)) \\
                             &= \varphi(q_{n*}(T,...,T,X,Z)) \\
                             &+ \varphi(q_{n*}(T,...,T,Y,Z)).
\end{aligned}
$$
In a similar way we have
$$
\begin{aligned}
\varphi(q_{n*}(T,...,T,Z,H)) &= \varphi(q_{n*}(T,...,T,Z,X)) \\
                                                   &+ \varphi(q_{n*}(T,...,T,Z,Y)).
\end{aligned}
$$

\end{proof}






\section{Main theorem}

We shall prove as follows a part of the the main result of this paper:

\begin{theorem}\label{mainthm1} 
Let $\A$ and $\A'$ be two $C^*$-algebras with identities $I_{\A}$ and $I_{\A'}$, respectively, and $P_1$ and $P_2 = I_{\A} - P_1$ nontrivial projections in $\A$. Suppose that $\A$ satisfies
\begin{eqnarray*}
  &&\left(\spadesuit\right) \ \ \  \ \ \  X \A P_i = \left\{0\right\} \ \ \  \mbox{implies} \ \ \ X = 0.
\end{eqnarray*}
Even more, suppose that $\varphi: \A \rightarrow \A'$ is a bijective unital map which satisfies
\begin{eqnarray*}
 &&\varphi(q_{n*}(I_\A,..., I_\A, A,B)) = q_{n*}(\varphi(I_\A),...., \varphi(I_\A), \varphi(A),\varphi(B))
\\&and&
\\&&\varphi(q_{n*}(P, ..., P, A, B)) = q_{n*}(\varphi(P), ..., \varphi(P),\varphi(A),\varphi(B)),
\end{eqnarray*}
 for all $A, B \in \A$ and $P \in \left\{P_1, P_2\right\}$. Then $\varphi$ is $*$-additive.
\end{theorem}




The following lemmas have the same hypotheses as the Theorem \ref{mainthm1} and we need them to prove the $*$-additivity of $\varphi$. 

\begin{lemma}\label{Claim1}  $\varphi(0) = 0$.
\end{lemma}

\begin{proof}
Since $\varphi$ is surjective, there exists $X \in \A$ such that $\varphi(X) = 0$. Firstly, consider that $n \geq 3$. Then,
$$
0 = q_{n*}(\varphi(P_1),...,\varphi(P_1),\varphi(P_2),\varphi(X)) = \varphi(q_{n*}(P_1,...,P_1,P_2,X)) = \varphi(0).
$$
Now, if $n=2$ we have
$$
\varphi(0) = \varphi(q_{2*}(I_\A,0)) = q_{2*}(\varphi(I_\A),\varphi(0)) = q_{2*}(I_{\A'},\varphi(0)) = 2\varphi(0).
$$
Therefore, $\varphi(0) = 0$.
\end{proof}

\begin{lemma}\label{lema1} For any $A_{11} \in \A_{11}$ and $B_{22} \in \A_{22}$, we have 
$$\varphi(A_{11} + B_{22}) = \varphi(A_{11}) + \varphi(B_{22}).$$
\end{lemma}
\begin{proof}
Since $\varphi$ is surjective, given $\varphi(A_{11})+\varphi(B_{22}) \in \A'$ there exists $T \in \A$ such that 
$\varphi(T) = \varphi(A_{11})+\varphi(B_{22})$, with $T=T_{11}+T_{12}+T_{21}+T_{22}$. Now, by Claim \ref{Claim2}
$$
\varphi(q_{n*}(P_i,...,P_i,T)) = \varphi(q_{n*}(P_i,...,P_i,A_{11})) + \varphi(q_{n*}(P_i,...,P_i,B_{22})),
$$
with $i=1,2$. It follows that
$$
\varphi(2^{n-2}(P_iT + TP_i)) = \varphi(2^{n-2}(P_iA_{11} + A_{11}P_i)) + \varphi(2^{n-2}(P_iB_{22} + B_{22}P_i)).
$$
Using the injectivity of $\varphi$ we obtain
$$
2^{n-2}(2T_{11} + T_{12} + T_{21}) = 2^{n-2}(2A_{11})
$$
and
$$
2^{n-2}(2T_{22} + T_{12} + T_{21}) = 2^{n-2}(2B_{22}).
$$
Then $T_{11}=A_{11}$, $T_{22}=B_{22}$ and $T_{12}=T_{21}=0$.
\end{proof}

\begin{lemma}\label{lema2}
For any $A_{12} \in \A_{12}$ and $B_{21} \in \A_{21}$, we have $\varphi(A_{12} + B_{21}) = \varphi(A_{12}) + \varphi(B_{21})$.
\end{lemma}

\begin{proof}
Since $\varphi$ is surjective, given $\varphi(A_{12})+\varphi(B_{21}) \in \A'$ there exists $T \in \A$ such that 
$\varphi(T) = \varphi(A_{12})+\varphi(B_{21})$, with $T=T_{11}+T_{12}+T_{21}+T_{22}$. Now, by Claim \ref{Claim2}
$$
\begin{aligned}
\varphi(q_{n*}(P_1,...,P_1,\frac{1}{2^{n-2}}P_1,T)) &= \varphi(q_{n*}(P_1,...,P_1,\frac{1}{2^{n-2}}P_1,A_{12})) \\&+ \varphi(q_{n*}(P_1,...,P_1,\frac{1}{2^{n-2}}P_1,B_{21})) \\
                               &= \varphi(P_1A_{12} + A_{12}P_1) + \varphi(P_1B_{21} + B_{21}P_1) \\
                               &= \varphi(A_{12}) + \varphi(B_{21}) = \varphi(T).
\end{aligned}
$$
Since $\varphi$ is injective,
$$
P_1T + TP_1 = T,
$$
that is,
$$
2T_{11} + T_{12} + T_{21} = T_{11}+T_{12}+T_{21}+T_{22}.
$$
Then $T_{11} = T_{22} = 0$.
Now, observe that, for $C_{12} \in \A_{12}$, $q_{n*}(P_1,...,P_1,A_{12},C_{12}) \in \A_{11}$ and $q_{n*}(P_1,...,P_1,B_{21},C_{12}) \in \A_{22}$.
Then, by Claim \ref{Claim2} and Lemma \ref{lema1}, we obtain
$$
\begin{aligned}
\varphi(q_{n*}(P_1,...,P_1,T,C_{12})) &= \varphi(q_{n*}(P_1,...,P_1,A_{12},C_{12})) + \varphi(q_{n*}(P_1,...,P_1,B_{21},C_{12})) \\
                                      &= \varphi(q_{n*}(P_1,...,P_1,A_{12},C_{12})+q_{n*}(P_1,...,P_1,B_{21},C_{12})).
\end{aligned}
$$
By injectivity of $\varphi$ we have
$$
q_{n*}(P_1,...,P_1,T,C_{12})=q_{n*}(P_1,...,P_1,A_{12},C_{12})+q_{n*}(P_1,...,P_1,B_{21},C_{12}),
$$
that is,
$$
T_{21}C_{12}+C_{12}T_{12}^* = B_{21}C_{12}+C_{12}A_{12}^*.
$$
Therefore,
$$
(T_{21} - B_{21})C_{12} = 0 \mbox{\,\, and \,\,} C_{12}(T_{12}^*-A_{12}^*) = 0.
$$
Finally, by $\left(\spadesuit\right)$ we conclude that $T_{12}=A_{12}$ and $T_{21}=B_{21}$.
\end{proof}

\begin{lemma}\label{lema3}
For any $A_{11} \in \A_{11}$, $B_{12} \in \A_{12}$, $C_{21} \in \A_{21}$ and $D_{22} \in \A_{22}$ we have 
$$\varphi(A_{11} + B_{12} + C_{21}) = \varphi(A_{11}) + \varphi(B_{12})+ \varphi(C_{21})$$
and
$$\varphi(B_{12} + C_{21} + D_{22}) = \varphi(B_{12})+ \varphi(C_{21}) + \varphi(D_{22}).$$
\end{lemma}

\begin{proof}
Since $\varphi$ is surjective, given $\varphi(A_{11})+\varphi(B_{12})+\varphi(C_{21}) \in \A'$ there exists $T \in \A$ such that 
$\varphi(T) = \varphi(A_{11})+\varphi(B_{12})+\varphi(C_{21})$, with $T=T_{11}+T_{12}+T_{21}+T_{22}$. Now, observing that
$q_{n*}(P_2,...,P_2,A_{11})=0$ and using Claim \ref{Claim2} and Lemma \ref{lema2}, we obtain
$$
\begin{aligned}
\varphi(q_{n*}(P_2,...,P_2,T)) &= \varphi(q_{n*}(P_2,...,P_2,A_{11})) + \varphi(q_{n*}(P_2,...,P_2,B_{12})) \\&+ \varphi(q_{n*}(P_2,...,P_2,C_{21})) \\
                               &= \varphi(q_{n*}(P_2,...,P_2,B_{12}) + q_{n*}(P_2,...,P_2,C_{21})).
\end{aligned}
$$
By injectivity of $\varphi$ we have
$$
q_{n*}(P_2,...,P_2,T)=q_{n*}(P_2,...,P_2,B_{12}) + q_{n*}(P_2,...,P_2,C_{21}),
$$
that is,
$$
2T_{22} + T_{12} + T_{21} = B_{12} + C_{21}.
$$
Therefore, $T_{22}=0$, $T_{12}=B_{12}$ and $T_{21}=C_{21}$.
Again, observing that $q_{n*}(I_\A,...,I_\A,P_1-P_2,B_{12})=q_{n*}(I_\A,...,I_\A,P_1-P_2,C_{21})=0$ and using 
Claim \ref{Claim2}, we obtain
$$
\begin{aligned}
\varphi(q_{n*}(I_\A,...,I_\A,P_1-P_2,T)) &= \varphi(q_{n*}(I_\A,...,I_\A,P_1-P_2,A_{11}))\\&+\varphi(q_{n*}(I_\A,...,I_\A,P_1-P_2,B_{12})) \\
                                         &+ \varphi(q_{n*}(I_\A,...,I_\A,P_1-P_2,C_{21})) \\
                                         &= \varphi(q_{n*}(I_\A,...,I_\A,P_1-P_2,A_{11})).
\end{aligned}
$$
By injectivity of $\varphi$ we have
$$
q_{n*}(I_\A,...,I_\A,P_1-P_2,T) = q_{n*}(I_\A,...,I_\A,P_1-P_2,A_{11}),
$$
that is,
$$
2T_{11} - 2T_{22} = 2A_{11}.
$$
Therefore, $T_{11}=A_{11}$.

The other identity we obtain in a similar way.
\end{proof}

\begin{lemma}\label{lema4}
For any $A_{11} \in \A_{11}$, $B_{12} \in \A_{12}$, $C_{21} \in \A_{21}$ and $D_{22} \in \A_{22}$ we have 
$$\varphi(A_{11} + B_{12} + C_{21} + D_{22}) = \varphi(A_{11}) + \varphi(B_{12}) + \varphi(C_{21}) + \varphi(D_{22}).$$
\end{lemma}
\begin{proof}
Since $\varphi$ is surjective, given $\varphi(A_{11})+\varphi(B_{12})+\varphi(C_{21})+\varphi(D_{22}) \in \A'$ there exists $T \in \A$ such that 
$\varphi(T) = \varphi(A_{11})+\varphi(B_{12})+\varphi(C_{21})+\varphi(D_{22})$, with $T=T_{11}+T_{12}+T_{21}+T_{22}$. Now, observing that
$q_{n*}(P_1,...,P_1,D_{22})=0$ and using Claim \ref{Claim2} and Lemma \ref{lema3}, we obtain
$$
\begin{aligned}
\varphi(q_{n*}(P_1,...,P_1,T)) &= \varphi(q_{n*}(P_1,...,P_1,A_{11})) + \varphi(q_{n*}(P_1,...,P_1,B_{12})) \\
                               &+ \varphi(q_{n*}(P_1,...,P_1,C_{21})) + \varphi(q_{n*}(P_1,...,P_1,D_{22})) \\
                               &= \varphi(q_{n*}(P_1,...,P_1,A_{11})) + \varphi(q_{n*}(P_1,...,P_1,B_{12})) \\&+ \varphi(q_{n*}(P_1,...,P_1,C_{21})) \\
                               &= \varphi(q_{n*}(P_1,...,P_1,A_{11}) + q_{n*}(P_1,...,P_1,B_{12}) \\&+ q_{n*}(P_1,...,P_1,C_{21})).
\end{aligned}
$$
By injectivity of $\varphi$ we have
$$
q_{n*}(P_1,...,P_1,T) = q_{n*}(P_1,...,P_1,A_{11}) + q_{n*}(P_1,...,P_1,B_{12}) + q_{n*}(P_1,...,P_1,C_{21}),
$$
that is,
$$
2T_{11} + T_{12} + T_{21} = 2A_{11} + B_{12} + C_{21}.
$$
Therefore, $T_{11}=A_{11}$, $T_{12}=B_{12}$ and $T_{21}=C_{21}$.

In a similar way, using $q_{n*}(P_2,...,P_2,T)$, we obtain
$$
2T_{22} + T_{12} + T_{21} = 2D_{22} + B_{12} + C_{21}
$$
and then $T_{22}=D_{22}$.
\end{proof}

\begin{lemma}\label{lema5}
For all $A_{ij}, B_{ij} \in \A_{ij}$, we have $\varphi(A_{ij} + B_{ij}) = \varphi(A_{ij}) + \varphi(B_{ij})$ for $i \neq j$.
\end{lemma}
\begin{proof}
By Lemma \ref{lema4} we have
\begin{equation*}
\begin{split}
\varphi(A_{ij} + B_{ij}) &+ \varphi(A_{ij}^*) + \varphi(B_{ij}A_{ij}^*) = \varphi(A_{ij} + B_{ij} + A_{ij}^* + B_{ij}A_{ij}^*) \\
                    &= \varphi(q_{n*}(P_i,...,P_i,\frac{1}{2^{n-2}}P_i+\frac{1}{2^{n-3}}A_{ij},P_j+B_{ij})) \\
                    &= q_{n*}(\varphi(P_i),...,\varphi(P_i),\varphi(\frac{1}{2^{n-2}}P_i+\frac{1}{2^{n-3}}A_{ij}),\varphi(P_j+B_{ij})) \\
                    &= q_{n*}(\varphi(P_i),...,\varphi(P_i),\varphi(\frac{1}{2^{n-2}}P_i)+\varphi(\frac{1}{2^{n-3}}A_{ij}),\varphi(P_j)+\varphi(B_{ij})) \\
                    &= \varphi(q_{n*}(P_i,...,P_i,\frac{1}{2^{n-2}}P_i,P_j)) \\
                    &+ \varphi(q_{n*}(P_i,...,P_i,\frac{1}{2^{n-2}}P_i,B_{ij})) \\
                    &+ \varphi(q_{n*}(P_i,...,P_i,\frac{1}{2^{n-3}}A_{ij},P_j)) \\
                    &+ \varphi(q_{n*}(P_i,...,P_i,\frac{1}{2^{n-3}}A_{ij},B_{ij})) \\
                    &= \varphi(B_{ij}) + \varphi(A_{ij} + A_{ij}^*) + \varphi(B_{ij}A_{ij}^*) \\
                    &= \varphi(B_{ij}) + \varphi(A_{ij}) + \varphi(A_{ij}^*) + \varphi(B_{ij}A_{ij}^*). 
\end{split}
\end{equation*}
Therefore,
$$
\varphi(A_{ij} + B_{ij}) = \varphi(A_{ij}) + \varphi(B_{ij}).
$$
\end{proof}

\begin{lemma}\label{lema6}
For all $A_{ii}, B_{ii} \in \A_{ii}$, we have $\varphi(A_{ii} + B_{ii}) = \varphi(A_{ii}) + \varphi(B_{ii})$ for $i \in \left\{1,2\right\}.$
\end{lemma}
\begin{proof}
Since $\varphi$ is surjective, given $\varphi(A_{ii})+\varphi(B_{ii}) \in \A'$, $i=1,2$, there exists $T \in \A$ such that 
$\varphi(T) = \varphi(A_{ii})+\varphi(B_{ii})$, with $T=T_{11}+T_{12}+T_{21}+T_{22}$. By Claim \ref{Claim2}, for $j \neq i$,
$$
\varphi(q_{n*}(P_j,...,P_j,T)) = \varphi(q_{n*}(P_j,...,P_j,A_{ii})) + \varphi(q_{n*}(P_j,...,P_j,B_{ii})) = 0.
$$
Then, $T_{ij}=T_{ji}=T_{jj}=0$. We just have to show that $T_{ii} = A_{ii} + B_{ii}$. Given $C_{ij} \in \A_{ij}$, using Lemma \ref{lema5} and  
Claim \ref{Claim2} we have
$$
\begin{aligned}
\varphi(q_{n*}(P_i,...,P_i,T,C_{ij})) &= \varphi(q_{n*}(P_i,...,P_i,A_{ii},C_{ij}))+\varphi(q_{n*}(P_i,...,P_i,B_{ii},C_{ij})) \\
                                      &= \varphi(q_{n*}(P_i,...,P_i,A_{ii},C_{ij}) + q_{n*}(P_i,...,P_i,B_{ii},C_{ij})).
\end{aligned}
$$
By injectivity of $\varphi$ we obtain
$$
q_{n*}(P_i,...,P_i,T,C_{ij}) = q_{n*}(P_i,...,P_i,A_{ii},C_{ij}) + q_{n*}(P_i,...,P_i,B_{ii},C_{ij}),
$$
that is,
$$
(T_{ii} - A_{ii} - B_{ii})C_{ij} = 0.
$$
Finally, by $\left(\spadesuit\right)$ we conclude that $T_{ii} = A_{ii} + B_{ii}$.
\end{proof}

Now we are able to show that $\varphi$ preserves $*$-addition.

Using Lemmas \ref{lema4}, \ref{lema5}, \ref{lema6} we have, for all $A,B \in \A$,
$$
\begin{aligned}
\varphi(A + B) &= \varphi(A_{11}+A_{12}+A_{21}+A_{22}+B_{11}+B_{12}+B_{21}+B_{22}) \\
               &= \varphi(A_{11}+B_{11})+\varphi(A_{12}+B_{12})+\varphi(A_{21}+B_{21})+\varphi(A_{22}+B_{22}) \\
               &= \varphi(A_{11})+\varphi(B_{11})+\varphi(A_{12})+\varphi(B_{12})+\varphi(A_{21})+\varphi(B_{21})+\varphi(A_{22})+\varphi(B_{22}) \\
               &= \varphi(A_{11}+A_{12}+A_{21}+A_{22}) + \varphi(B_{11}+B_{12}+B_{21}+B_{22}) = \varphi(A) + \varphi(B).
\end{aligned}
$$
Besides, on the one hand, since $\varphi$ is additive it follows that 
$$\varphi(A + A^{*}) = \varphi(A) + \varphi(A^{*}).$$
On the other hand, by additivity of $\varphi$,
$$
\begin{aligned}
2^{n-2}\varphi(A + A^{*}) &= \varphi(2^{n-2}(A + A^{*}))  = \varphi(q_{n*}(I_{\A},...,I_\A,A,I_{\A})) \\
               &= q_{n*}(I_{\A'},...,I_{\A'},\varphi(A),I_{\A'}) = 2^{n-2}(\varphi(A) + \varphi(A)^{*}).
\end{aligned}
$$
Therefore $\varphi(A^{*}) = \varphi(A)^{*}$ and Theorem \ref{mainthm1} is proved.

\vspace{0,5cm}
Now we focus our attention on investigate the problem of when $\varphi$ is a $*$-ring isomorphism. 
We prove the following result:
\\

\begin{theorem}\label{mainthm2}
Let $\A$ and $\A'$ be two  
$C^*$-algebras with identities $I_{\A}$ and $I_{\A'}$, respectively, and $P_1$ and $P_2 = I_{\A} - P_1$ nontrivial 
projections in $\A$. Suppose that $\A$ and $\A'$ satisfy:
\begin{eqnarray*}
  &&\left(\spadesuit\right) \ \ \  \ \ \  X \A P_i = \left\{0\right\} \ \ \  \mbox{implies} \ \ \ X = 0 
  \\
  \mbox{ and}
	\\&& \left(\clubsuit \right) \ \ \  \ Y \A' \varphi(P_i) = \left\{0\right\} \ \ \  \mbox{implies} \ \ \ Y = 0.
\end{eqnarray*}
If $\varphi: \A \rightarrow \A'$ is a bijective unital map which satisfies
\begin{eqnarray*}
 &&\varphi(q_{n*}(I_\A,..., I_\A, A,B)) = q_{n*}(\varphi(I_\A),...., \varphi(I_\A), \varphi(A),\varphi(B))
\\&and&
\\&&\varphi(q_{n*}(P, ..., P, A, B)) = q_{n*}(\varphi(P), ..., \varphi(P),\varphi(A),\varphi(B)),
\end{eqnarray*}
 for all $A, B \in \A$ and $P \in \left\{P_1, P_2\right\}$ then $\varphi$ is $*$-ring isomorphism.
 
\end{theorem}

Since $\varphi$ is $*$-additive, by Theorem \ref{mainthm1},  it is enough to verify that $\varphi(AB) = \varphi(A)\varphi(B)$. Firstly, let us prove the following lemmas:

\begin{lemma}\label{lemam0}
 $Q_i = \varphi(P_i)$ is a projection in $\A'$, with $i \in \{1,2\}$.
\end{lemma}
\begin{proof}
 By additivity of $\varphi$ we have
 $$
 \begin{aligned}
 2^{n-1}Q_i &= 2^{n-1}\varphi(P_i) = \varphi(2^{n-1}P_i) \\
            &= \varphi(q_{n*}(I_\A, ...,I_\A,P_i,P_i)) \\
            &= q_{n*}(I_{\A'}, ...,I_{\A'},\varphi(P_i),\varphi(P_i)) \\
            &= 2^{n-1}\varphi(P_i)\varphi(P_i) = 2^{n-1}Q_iQ_i.
 \end{aligned}
 $$
 Therefore, $Q_iQ_i = Q_i$.
\end{proof}

\begin{lemma}\label{lemam1}
If $X \in \A_{ij}$ then $\varphi(X) \in \A'_{ij}$.
\end{lemma}
\begin{proof}
Firstly, given $X \in \A_{ij}$, with $i \neq j$, we observe that
$$
\begin{aligned}
2^{n-2}\varphi(X) &= \varphi(2^{n-2}X) = \varphi(q_{n*}(P_j,...,P_j,X)) = q_{n*}(\varphi(P_j),...,\varphi(P_j),\varphi(X)) \\
                  &= 2^{n-2}(Q_j\varphi(X) + \varphi(X)Q_j),
\end{aligned}
$$
that is, $Q_i\varphi(X)Q_i = Q_j\varphi(X)Q_j = 0$. Even more,
$$
\begin{aligned}
0 &= \varphi(q_{n*}(P_i,...,P_i,X,P_i)) = q_{n*}(Q_i,...,Q_i,\varphi(X),Q_i) \\
  &= 2^{n-3}(Q_i\varphi(X)Q_i + \varphi(X)Q_i + Q_i\varphi(X)^*Q_i + Q_i\varphi(X)^*).
\end{aligned}
$$
Multiplying left side by $Q_j$ we obtain $Q_j\varphi(X)Q_i = 0$. Therefore, $\varphi(X) \in \A'_{ij}$. 
In a similar way, if $X \in \A_{ii}$ we conclude that $\varphi(X) \in \A'_{ii}$.
\end{proof}

\begin{lemma}\label{lemam2}
If $A_{ii} \in \A_{ii}$ and $B_{ij} \in \A_{ij}$, with $i \neq j$, then $\varphi(A_{ii}B_{ij})=\varphi(A_{ii})\varphi(B_{ij})$.
\end{lemma}
\begin{proof}
Let $A_{ii} \in \A_{ii}$ and $B_{ij} \in \A_{ij}$, with $i \neq j$. Then, by Lemma \ref{lemam1} and additivity of $\varphi$,
$$
\begin{aligned}
2^{n-2}\varphi(A_{ii}B_{ij}) &= \varphi(2^{n-2}A_{ii}B_{ij}) = \varphi(q_{n*}(P_i,...,P_i,A_{ii},B_{ij})) \\
            &= q_{n*}(\varphi(P_i),...,\varphi(P_i),\varphi(A_{ii}),\varphi(B_{ij})) = 2^{n-2}\varphi(A_{ii})\varphi(B_{ij}).
\end{aligned}
$$
Therefore,
$$
\varphi(A_{ii}B_{ij}) = \varphi(A_{ii})\varphi(B_{ij}).
$$
\end{proof}

\begin{lemma}\label{lemam3}
If $A_{ii},B_{ii} \in \A_{ii}$ then $\varphi(A_{ii}B_{ii}) = \varphi(A_{ii})\varphi(B_{ii})$.
\end{lemma}
\begin{proof}
Let $X$ be an element of $\A_{ij}$, with $i \neq j$. Using Lemma \ref{lemam2} we obtain
$$
\varphi(A_{ii}B_{ii})\varphi(X) = \varphi(A_{ii}B_{ii}X) = \varphi(A_{ii})\varphi(B_{ii}X) = \varphi(A_{ii})\varphi(B_{ii})\varphi(X),
$$
that is,
$$
(\varphi(A_{ii}B_{ii}) - \varphi(A_{ii})\varphi(B_{ii}))\varphi(X) = 0.
$$
Now, by Lemma \ref{lemam1}, since $\varphi(X) \in \A'_{ij}$ and $\varphi(A_{ii}B_{ii}) - \varphi(A_{ii})\varphi(B_{ii}) \in \A'_{ii}$, we have
$$
(\varphi(A_{ii}B_{ii}) - \varphi(A_{ii})\varphi(B_{ii}))\A'\varphi(P_j) = \{0\}.
$$
Finally, $\left(\clubsuit \right)$ ensures that $\varphi(A_{ii}B_{ii}) = \varphi(A_{ii})\varphi(B_{ii})$.
\end{proof}

\begin{lemma} \label{lemam4}
If $A_{ij} \in \A_{ij}$ and $B_{ji} \in \A_{ji}$, with $i \neq j$, then $\varphi(A_{ij}B_{ji}) = \varphi(A_{ij})\varphi(B_{ji})$.
\end{lemma}
\begin{proof}
Let $A_{ij} \in \A_{ij}$ and $B_{ji} \in \A_{ji}$, with $i \neq j$. Then, by Lemma \ref{lemam1} and additivity of $\varphi$, 
\end{proof}
$$
\begin{aligned}
2^{n-3}\varphi(A_{ij}B_{ji}) &= \varphi(2^{n-3}A_{ij}B_{ji}) = \varphi(q_{n*}(P_i,...,P_i,A_{ij},B_{ji})) \\
     &= q_{n*}(\varphi(P_i),...,\varphi(P_i),\varphi(A_{ij}),\varphi(B_{ji})) = 2^{n-3}\varphi(A_{ij})\varphi(B_{ji}).
\end{aligned}
$$
Therefore,
$$
\varphi(A_{ij}B_{ji}) = \varphi(A_{ij})\varphi(B_{ji}).
$$

\begin{lemma}\label{lemam5}
If $A_{ij} \in \A_{ij}$ and $B_{jj} \in \A_{jj}$, with $i \neq j$, then $\varphi(A_{ij}B_{jj}) = \varphi(A_{ij})\varphi(B_{jj})$
\end{lemma}
\begin{proof}
Let $X$ be an element of $\A_{ji}$, with $i \neq j$. Using Lemmas \ref{lemam2} and \ref{lemam4} we obtain
$$
\varphi(A_{ij}B_{jj})\varphi(X) = \varphi(A_{ij}B_{jj}X) = \varphi(A_{ij})\varphi(B_{jj}X) = \varphi(A_{ij})\varphi(B_{jj})\varphi(X),
$$
that is,
$$
(\varphi(A_{ij}B_{jj}) - \varphi(A_{ij})\varphi(B_{jj}))\varphi(X) = 0.
$$
Now, by Lemma \ref{lemam1}, since $\varphi(X) \in \A'_{ji}$ and $\varphi(A_{ij}B_{jj}) - \varphi(A_{ij})\varphi(B_{jj}) \in \A'_{ij}$, we have
$$
(\varphi(A_{ij}B_{jj}) - \varphi(A_{ij})\varphi(B_{jj}))\A'\varphi(P_i) = \{0\}.
$$
Finally, $\left(\clubsuit \right)$ ensures that $\varphi(A_{ij}B_{jj}) = \varphi(A_{ij})\varphi(B_{jj})$.
\end{proof}

Thus, by additivity of $\varphi$, proved in the Theorem $\ref{mainthm1}$, and the lemmas above we conclude that $\varphi(AB) = \varphi(A)\varphi(B)$. Therefore $\varphi$ is a $*$-ring isomorphism.

\section{Corollaries} 

Let us present some consequences of the our main result. The first one provides the conjecture that appears in \cite{Ferco} to the case of multiplicative $*$-Jordan-type maps:

\begin{corollary}   
Let $\A$ and $\A'$ be two  
$C^*$-algebras with identities $I_{\A}$ and $I_{\A'}$, respectively, and $P_1$ and $P_2 = I_{\A} - P_1$ 
nontrivial projections in $\A$. Suppose that $\A$ and $\A'$ satisfy:
\begin{eqnarray*}
  &&\left(\spadesuit\right) \ \ \  \ \ \  X \A P_i = \left\{0\right\} \ \ \  \mbox{implies} \ \ \ X = 0 
  \\
  \mbox{ and}
	\\&& \left(\clubsuit \right) \ \ \  \ Y \A' \varphi(P_i) = \left\{0\right\} \ \ \  \mbox{implies} \ \ \ Y = 0.
\end{eqnarray*}
Then $\varphi: \A \rightarrow \A'$ is a bijective unital multiplicative $*$-Jordan $n$-map if and only if $\varphi$ is a
$*$-ring isomorphism.
\end{corollary}

Observing that prime $C^*$-algebras satisfy $(\spadesuit), (\clubsuit)$ we have the following result: 

\begin{corollary} 
Let $\A$ and $\A'$ be prime $C^*$-algebras with identities $I_{\A}$ and $I_{\A'}$, respectively, and $P_1$ and $P_2 = I_{\A} - P_1$ nontrivial projections in $\A$. Then $\varphi: \A \rightarrow \A'$ is a bijective unital multiplicative $*$-Jordan $n$-map if and only if $\varphi$ is a $*$-ring isomorphism.
\end{corollary}

A von Neumann algebra $\mathcal{M}$ is a weakly closed, self-adjoint algebra of operators on a Hilbert space $\mathcal{H}$ containing the identity operator $I$. As an application on von Neumann algebras we have the following: 

\begin{corollary}
Let $\mathcal{M}$ be a von Neumann algebra without central summands of type $I_1$. Then $\varphi: \mathcal{M} \rightarrow \mathcal{M}$ is a bijective unital multiplicative $*$-Jordan $n$-map if and only if $\varphi$ is a $*$-ring isomorphism.
\end{corollary}
\begin{proof}
Let $\mathcal{M}$ be the von Neumann algebra. It is shown in \cite{Bai} and \cite{Mie} that if a von Neumann algebra has no central summands of type $I_1$, then $\mathcal{M}$ satifies the following assumption: 
\begin{itemize}
\item $X \mathcal{M}P_i = \left\{0\right\} \Rightarrow X = 0$.
\end{itemize}
Thus, by Theorem \ref{mainthm2} the corollary is true.
\end{proof}

 To finish, $\mathcal{M}$ is a factor von Neumann algebra if its center only contains the scalar operators. It is well known that a factor von Neumann algebra is prime and then we have the following:

\begin{corollary}
Let $\mathcal{M}$ be a factor von Neumann algebra. Then $\varphi: \mathcal{M} \rightarrow \mathcal{M}$ is a bijective unital multiplicative $*$-Jordan $n$-map if and only if $\varphi$ is a $*$-ring isomorphism.
\end{corollary}

\end{document}